\documentclass{amsart}
\usepackage{a4wide,mathtools,amssymb}

\usepackage[all]{xy}
\usepackage{color}

\newcommand{\IN}{\mathbb N}
\newcommand{\IR}{\mathbb R}
\newcommand{\w}{\omega}
\newcommand{\IQ}{\mathbb Q}
\newcommand{\IZ}{\mathbb Z}
\newcommand{\defeq}{\coloneqq}
\newcommand{\updownarrows}{\;{\uparrow}\!{\downarrow}\;}
\newcommand{\dist}{\mathsf{d}}
\newcommand{\Sf}{{\mathsf S}}

\newcommand{\dom}{\mathsf{dom}}
\newcommand{\rng}{\mathsf{rng}}

\newtheorem{theorem}{Theorem}[section]
\newtheorem{corollary}[theorem]{Corollary}
\newtheorem{example}[theorem]{Example}
\newtheorem{lemma}[theorem]{Lemma}
\newtheorem{claim}[theorem]{Claim}

\newtheorem{proposition}[theorem]{Proposition}
\newtheorem{problem}[theorem]{Problem}

\theoremstyle{definition}

\newtheorem{definition}[theorem]{Definition}
\newtheorem{remark}[theorem]{Remark}

\title{Metric characterizations of some subsets of the real line}
\author{Iryna Banakh, Taras Banakh, Maria Kolinko, Alex Ravsky}
\address{I.~Banakh, A.~Ravsky: Institute for Applied Problems of Mechanics and Mathematics of National Academy of Sciences of Ukraine,  Naukova 3b, Lviv, Ukraine}
\email{ibanakh@yahoo.com, alexander.ravsky@uni-wuerzburg.de}
\address{T.~Banakh, M.~Kolinko: Ivan Franko National University of Lviv, Ukraine}
\email{t.o.banakh@gmail.com, marybus20@gmail.com}
\subjclass[2010]{51F99, 54E35, 54E40}
\keywords{Metric space, Triangle Equality, subline, ray}

\begin{document}
\begin{abstract} A metric space $(X,\dist)$ is called a {\em subline} if every 3-element subset $T$ of $X$ can be written as $T=\{x,y,z\}$ for some points $x,y,z$ such that $\dist(x,z)=\dist(x,y)+\dist(y,z)$. By a classical result of Menger, every subline of cardinality $\ne 4$ is isometric to a subspace of the real line. A subline $(X,\dist)$ is called a {\em $n$-subline} for a natural number $n$ if for every $c\in X$ and positive real number $r\in\dist[X^2]$, the sphere $\Sf(c;r)\defeq\{x\in X:\dist(x,c)=r\}$ contains at least $n$ points. We prove that every $2$-subline is isometric to some additive subgroup of the real line. Moreover, for every subgroup $G\subseteq\IR$, a metric space $(X,\dist)$ is isometric to $G$ if and only if $X$ is a $2$-subline with $\dist[X^2]=G_+\defeq G\cap[0,\infty)$. A metric space $(X,\dist)$ is called a {\em ray} if $X$  is a $1$-subline and $X$ contains a point $o\in X$ such that for every $r\in\dist[X^2]$ the sphere $\Sf(o;r)$ is a singleton. We prove that for a subgroup $G\subseteq\IQ$, a metric space $(X,\dist)$ is isometric to the ray $G_+$ if and only if $X$ is a ray with $\dist[X^2]=G_+$. A metric space $X$ is isometric to the ray $\IR_+$ if and only if $X$ is a complete ray such that $\IQ_+\subseteq \dist[X^2]$. On the other hand, the real line contains a dense ray $X\subseteq\IR$ such that $\dist[X^2]=\IR_+$.
\end{abstract}

\maketitle

\section{Introduction and Main Results}

In this paper we discuss characterizations of metric spaces which are isometric to some important subspaces of the real line, in particular, to the spaces $\IN,\IZ,\IQ,\IR$ of natural, integer, rational, real numbers, respectively. The space of real numbers $\IR$ and its subspaces $\IN,\IZ,\IQ$ are endowed with the standard Euclidean metric $\dist(x,y)=|x-y|$. For a subset $X\subseteq\IR$, let $X_+\defeq \{x\in X:x\ge 0\}$.

The sets $\IZ$ and $\IQ$ are subgroups of the real line, and $\IZ_+,\IQ_+,\IR_+$ are submonoids of $\IR$.
%The subsets $\IZ$ and $\IQ$ are subgroups of the {\red group $(\IR,+)$}, and $\IZ_+,\IQ_+,\IR_+$ are submonoids of {\red $(\IR,+)$}.

A set $X\subseteq \IR$ is called
\begin{itemize}
\item a {\em submonoid} of $\IR$ if $0\in X$ and $x+y\in X$ for all $x,y\in X$;
\item a {\em subgroup} of $\IR$ if $X$ is a submonoid of $\IR$ such that $-x\in X$ for every $x\in X$.
\end{itemize}

For a metric space $X$, we denote by $\dist_X$ (or just by $\dist$ if $X$ is clear from the context) the metric of the space $X$. For two points $x,y$ of a metric space $X$, the real number $\dist_X(x,y)$ will be denoted by $xy$.

Two metric spaces $X$ and $Y$ are {\em isometric} if there exists a bijective function $f:X\to Y$ such that $\dist_Y(f(x),f(y))=\dist_X(x,y)$ for all $x,y\in X$. A metric space $X$ is defined to {\em embed\/} into a metric space $Y$ if $X$ is isometric to some subspace of $Y$.

Observe that the  space $\IN\defeq\IZ_+\setminus\{0\}$ is isometric to $\IZ_+$.

\begin{definition} A metric space $X$ is called a {\em subline} if any 3-element subset $T\subseteq X$ embeds into the real line. This happens if and only if any points $x,y,z\in X$ satisfy the following property called the {\em Triangle Equality}:
$yz=yx+xz\;\vee\;xz=xy+yz\;\vee\;xy=xz+zy.$
%$$x,y)=x,z)+z,y)\;\vee\;x,y)=x,z)+z,y)\;\vee\;x,y)=x,z)+z,y).$$
\end{definition}

According to an old result of Menger \cite{Menger} (see also \cite{Bowers}), a subline $X$ embeds into a real line if and only if it is not an $\ell_1$-rectangle.

\begin{definition}\label{d:rectangle} A metric space $(X,\dist)$ is called an {\em  $\ell_1$-rectangle} if $X=\{a,b,c,d\}$ for some pairwise distinct points $a,b,c,d$ such that $ab=cd$, $bc=ad$ and $ac=ab+bc=bd$.
%An $\ell_1$-rectangle $\{a,b,c,d\}$ is called an {\em $\ell_1$-square} if $ab=bc$.
\end{definition}

\begin{example} Let $\IR^2$ be the real plane endowed with the $\ell_1$-metric
$$\dist:\IR^2\times\IR^2\to\IR,\quad\dist:\big((x,y),(u,v)\big)\mapsto|x-u|+|y-v|.$$
 For any positive real numbers $a,b$, the subset $$\square_a^b\defeq \{(a,b),(a,-b),(-a,b),(-a,-b)\}$$ of  $\IR^2$ is an $\ell_1$-rectangle. Moreover, every $\ell_1$-rectangle is isometric to the $\ell_1$-rectangle $\square_a^b$ for unique positive real numbers $a\le b$.
\end{example}

 %Definition~\ref{d:rectangle} implies that every $\ell_1$-rectangle is a $1$-subline (but fails to be a $2$-subline).

 The following metric characterization of subspaces of the real line was surely known to Karl Menger \cite{Menger} and was also mentioned (without  proof) in \cite{Bowers}.

\begin{theorem}\label{t:main1} A metric space $X$ embeds into the real line if and only if $X$ is a subline and $X$ is not an $\ell_1$-rectangle.
\end{theorem}

Theorem~\ref{t:main1} has the following corollary.

\begin{corollary}\label{c:1} A metric space $X$ of cardinality $|X|\ne 4$ embeds into the real line if and only if it is a subline.
\end{corollary}

Corollary~\ref{c:1} is a partial case of the following characterization that was proved by Karl Menger \cite{Menger} in general terms of congruence relations and reproved  by John Bowers and Philip Bowers \cite{Bowers} for metric spaces.

\begin{theorem}\label{t:Rn} For every natural number $n$, a metric space $X$ of cardinality $|X|\ne n+3$ embeds into the Euclidean space $\IR^n$ if and only if every subspace $A\subseteq X$ of cardinality $|A|\le n+2$ embeds into $\IR^n$.
\end{theorem}

The paper \cite{Bowers} contains a decription of metric spaces of cardinality $n+3$ that do not embed into $\IR^n$ but whose all proper subsets do embed into $\IR^n$. For $n=1$ such metric spaces are exactly $\ell_1$-rectangles.

Theorem~\ref{t:main1} will be applied in the metric characterizations of the spaces $\IZ,\IQ,\IR,\IZ_+,\IQ_+,\IR_+$. Those characterizations involve the following definition.

\begin{definition}  Let $\kappa$ be a cardinal number. A metric space $X$ is called
\begin{itemize}
\item {\em $\kappa$-spherical} if for every $r\in\dist[X^2]\setminus\{0\}$ and $c\in X$ the sphere $\Sf(c;r)\defeq\{x\in X:xc=r\}$ contains at least $\kappa$ points;
\item a {\em $\kappa$-subline} if $X$ is a $\kappa$-spherical subline.
\end{itemize}
\end{definition}

For $\kappa>2$, the definition of a $\kappa$-subline is vacuous: indeed, assuming that some sphere $\Sf(c;r)$ in a subline contains three pairwise distinct points $x,y,z$,  we can apply the Triangle Equality and conclude that $xy=xc+cy=2r=xz=yz$, witnessing that the Triangle Equality fails for the points $x,y,z$.

Therefore, for every metric space we have the implications
$$\mbox{$2$-subline}\Rightarrow\mbox{$1$-subline}\Rightarrow\mbox{$0$-subline}\Leftrightarrow\mbox{subline}.$$

\begin{theorem}\label{t:main2} Every nonempty $2$-subline is isometric to a  subgroup of the real line. Moreover, a metric space $X$ is isometric to a subgroup $G$ of $\IR$ if and only if $X$ is a $2$-subline such that $\dist[X^2]=G_+$.
\end{theorem}

A metric space $X$ is called {\em Banakh} if for every $c\in X$ and $r\in\dist[X^2]$, there exist points $x,y\in X$ such that $\Sf(x;r)=\{x,y\}$ and $\dist(x,y)=2r$. It is easy to see that every 2-subline is a Banakh space. Theorem~\ref{t:main2} can be compared with the following metric characterizations of subgroups of $\IQ$, proved in \cite{Ban}.

\begin{theorem}\label{t:Ban} A metric space $X$ is isometric to a subgroup $G$ of the group $\IQ$ if and only if $X$ is a Banakh space with $\dist[X^2]=G_+$.
\end{theorem}

Theorems~\ref{t:main2} and \ref{t:Ban} imply the following characterizations of the  metric spaces $\IZ,\IQ,\IR$.

\begin{corollary} A metric space $X$ is isometric to $\IZ$ if and only if $X$ is a $2$-subline with $\dist[X^2]=\IZ_+$ if and only if $X$ is a Banakh space with $\dist[X^2]=\IZ_+$.
\end{corollary}

\begin{corollary} A metric space $X$ is isometric to $\IQ$ if and only if $X$ is a $2$-subline with $\dist[X^2]=\IQ_+$ if and only if $X$ is a Banakh space with $\dist[X^2]=\IQ_+$.
\end{corollary}

\begin{corollary}\label{c:R} A metric space $(X,\dist)$ is isometric to $\IR$ if and only if $X$ is a $2$-subline such that $\dist[X^2]=\IR_+$.
\end{corollary}

Corollary~\ref{c:R} can be compared with the following metric characterization of the real line, proved by Will Brian \cite{WB} (see also \cite[1.7]{Ban}).

\begin{theorem} A metric space $X$ is isometric to the real line if and only if $X$ is a complete Banakh space with $\IQ_+\subseteq\dist[X^2]$.
\end{theorem}

We recall that a metric space $X$ is {\em complete} if every Cauchy sequence in $X$ is convergent.
\smallskip

Metric characterizations of the spaces $\IZ_+,\IQ_+,\IR_+$ are based on the notion of a ray.

\begin{definition} A metric space $(X,\dist)$ is called a {\em ray} if $X$ is a $1$-subline containing a point  $o\in X$ such that for every $r\in\dist[X^2]$ the sphere $\Sf(o;r)$ is a singleton.
\end{definition}

Observe that no ray is a $2$-subline. %Let us recall \cite[??]{??} that the {\em rank} of an abelian group $G$ is the largest cardinality of a subset $B\subseteq G$ such that for any pairwise distinct elements $g_1,\dots,g_n\in G$ and any integer numbers $z_1,\dots,z_n$, the equality $z_1g_1+\dots+z_ng_n=0$ implies $z_1g_1=\dots=z_ng_n=0$.

\begin{theorem}\label{t:main3} Let $G$ be a subgroup of the additive group $\IQ$ of rational numbers. A metric space $X$ is isometric to $G_+$ if and only if $X$ is a ray with $\dist[X^2]=G_+$.
\end{theorem}

\begin{corollary} A metric space $X$ is isometric to $\IZ_+$ if and only if $X$ is a ray with $\dist[X^2]=\IZ_+$.
\end{corollary}

\begin{corollary} A metric space $X$ is isometric to $\IQ_+$ if and only if $X$ is a ray with $\dist[X^2]=\IQ_+$.
\end{corollary}

\begin{theorem}\label{t:main4} A metric space $X$ is isometric to $\IR_+$ if and only if $X$ is a complete ray with $\IQ_+\subseteq \dist[X^2]$.
\end{theorem}

The completeness cannot be removed from Theorem~\ref{t:main4} as shown by the following example.

\begin{example}\label{ex:1} For every subgroup $G\subseteq \IR$ containing nonzero elements $a,b\in G$ such that $b\notin \IQ{\cdot}a\subseteq G$, there exists a dense submonoid $X$ of $\IR$ such that $X$ is a ray with $d[X^2]=G_+$ and $X$ is not isometric to $G_+$.
\end{example}

Example~\ref{ex:1} shows that (in contrast to Theorem~\ref{t:main3}) Theorem~\ref{t:main4} does not hold for arbitrary subgroups of the real line. The submonoid $X$ in Example~\ref{ex:1} is the image of $G_+$ under a suitable additive bijective function $\Phi:G\to G$. A function $\Phi:G\to G$ on a group $G$ is {\em additive} if $\Phi(x+y)=\Phi(x)+\Phi(y)$ for all $x,y\in G$. Example~\ref{ex:1} suggests the following open

\begin{problem}\label{t:ray-field}   Is every ray $X$ with $\dist[X^2]=\IR_+$ isometric to the metric subspace $\Phi[\IR_+]$ of $\IR$ for some injective additive function $\Phi:\IR\to \IR$?
\end{problem}

By Theorem~\ref{t:main2}, every $2$-subline is isometric to a subgroup of the real line. In this context it would be interesting to know a classification of $1$-sublines. Observe that a metric subspace $X$ of the real line is a $1$-subline if and only if $X$ is $1$-spherical if and only if $X$ is semiaffine in the group $\IR$. A subset $X$ of an Abelian group $G$ is called {\em semiaffine} if for every $x,y,z\in X$ the doubleton $\{x+y-z,x-y+z\}$ intersects $X$. Semiaffine sets in Abelian groups were characterized in \cite{BBKR2} as follows.

\begin{theorem}\label{t:semiaffine} A subset $X$ of an Abelian group $G$ is semiaffine if and only if one of the following conditions holds:
\begin{enumerate}
\item $X=(H+a)\cup (H+b)$ for some subgroup $H$ of $G$ and some elements $a,b\in X$;
\item $X=(H\setminus C)+g$ for some $g\in G$, some subgroup $H\subseteq G$ and some midconvex set $C$ in $H$.
\end{enumerate}
\end{theorem}

A subset $X$ of a group $G$ is called {\em midconvex} in $G$ if for every $x,y\in X$ the set $$\frac{x+y}2\defeq\{z\in G:2z=x+y\}$$ is a subset of $X$.

The following characterization of $1$-sublines follows from Theorems~\ref{t:main1} and \ref{t:semiaffine}.

\begin{theorem}\label{t:1sub} A metric space $X$ is a $1$-subline if and only if $X$ is isometric to one of the following metric spaces:
\begin{enumerate}
\item the $\ell_1$-rectangle $\square_a^b$ for some positive real numbers $a,b$;
\item $(H+a)\cup(H+b)$ for some subgroup $H$ of $\IR$ and some real numbers $a,b$;
\item $H\setminus C$ for some subgroup $H$ of $\IR$ and some midconvex set $C$ in the group $H$.
\end{enumerate}
\end{theorem}

 Midconvex sets in  Abelian groups were characterized in \cite{BBKR1} as follows.

\begin{theorem}\label{t:1} A subset $X$ of an Abelian group $G$ is midconvex if and only if for every $g\in G$ and $x\in X$, the set $\{n\in\IZ:x+ng\in X\}$ is equal to $C\cap H$ for some order-convex set $C\subseteq \IZ$ and some subgroup $H\subseteq \IZ$ such that the quotient group $\IZ/H$ has no elements of even order.
\end{theorem}

A subset $C$ of a subgroup $H$ of $\IR$ is called {\em order-convex} in $H$ if for any $x,y\in C$, the order interval $\{z\in H:x\le y\le z\}$ is a subset of $C$.

Midconvex sets in subgroups of the group $\IQ$ were characterized in \cite{BBKR1} as follows.

\begin{theorem}\label{t:3} Let $H$ be a subgroup of $\IQ$. A nonempty set $ X\subseteq H$ is midconvex in $H$ if and only if $X=C\cap(P+x)$ for some order-convex set $C\subseteq H$, some $x\in X$ and some subgroup $P$ of $H$ such that the quotient group $H/P$ contains no elements of even order.
\end{theorem}

The necessary information on metric spaces can be found in \cite[Ch.4]{Eng}; for basic notions of group theory, we refer the reader to the textbook \cite{Rob}.

\section{Proof of Theorem~\ref{t:main1}}

Since we have found no published proof of Theorem~\ref{t:main1}, we present the detailed proof of this theorem in this section. Bowers and Bowers write in \cite{Bowers} that Theorem~\ref{t:main1} ``can be proved by chasing around betweenness relations among four points of $X$''. This indeed can be done with the help of the following lemma.

\begin{lemma}\label{l:Ravsky} If a finite subline $X$ is not an $\ell_1$-rectangle, then $X$ is isometric to a subspace of the real line.
\end{lemma}

\begin{proof} If $|X|\le 1$, then $X$ is isometric to a subspace of any nonempty metric space, including the real line. So, we assume that $|X|>1$. Since $X$ is finite, there exist points $a,b\in X$ such that $ab=D\defeq\max\{xy:x,y\in X\}$.  For every point $x\in X$, the maximality of $ab=D$ and the Triangle Equality for the points $\{a,x,b\}$ ensure that
\begin{equation}\label{eq:axb}
ab=ax+xb.
\end{equation}

We claim that the function $f:X\to\IR$, $f:x\mapsto ax$, is an isometric embedding of $X$ into the real line.

Indeed, otherwise there exist points $x,y\in X$ such that $xy\ne |ax-ay|$.
Then $x\ne a\ne y$ and the Triangle Equality for the points $x,y,a$ implies that $xa+ay=xy$.
Then $yx+xb+by=ya+ax+xb+by=2D$, so the triangle $\{x,y,b\}$ has a side of length $D$.
Taking into account that $x\ne a\ne y$, $xb=ab-ax<D$ and $yb=ab-ay<D$, we conclude that $xy=D$.
Then $xa=xy-ay=xy+yb-ab=yb$ and $xb=xy-by=xy-xa=ay$, which means that $\{x,a,y,b\}$ is an $\ell_1$-rectangle. It follows from $ax+xb=ab=ay+yb$ and $\{x,y\}\cap\{a,b\}=\emptyset$ that
\begin{equation}\label{eq1}
\max\{ax,xb,ay,yb\}<D.
\end{equation}

Since $X$ is not an $\ell_1$-rectangle, there exists a point $z\in X\setminus\{x,a,y,b\}$.  The Triangle Equality for the triangles $\{x,y,z\}$ and $\{a,z,b\}$ implies  $xz+zy=xy=D=ab=az+bz$. Consequently,
\begin{equation}\label{eq2}
\max\{xz,zy,az,bz\}<D.
\end{equation}

By the strict inequalities (\ref{eq1}) and (\ref{eq2}) and the Triangle Equality, no side of the triangles $\{x,z,b\}$,
$\{y,z,a\}$, and $\{y,z,b\}$ has length $D$. Then, by (\ref{eq:axb}) and the equalities $xz+yz=xy=xa+ay$,
$$
\begin{aligned}
xa+az-xz&=xa+az+zb+bx-(xz+zb+bx)>2ab-2D=0,\\
ax+xz-az&=ax+xz+yz+ay-(yz+za+ay)>2xy-2D=0,\\
az+zx-xa&=az+zx-xa+yz+zb+by-(yz+zb+by)\\
&=az+zx+yz+zb-(yz+zb+by)>ab+xy-2D=0.
\end{aligned}
$$
This contradicts the Triangle Equality for the points $a,x,z$.
%Therefore $xy=|ax-ay|$, a contradiction.}
%If $xz=xa+az$, then $xz+zb+xb=xa+az+zb+xb=2ab$,
%and by the Triangle Equality, one side of the triangle $\{x,z,b\}$ has length $D$, which contradicts the strict inequalities (\ref{eq1}) or (\ref{eq2}).
%If $az=ax+xz$ then $yz+za+ay=yz+ax+xz+ay=2xy$,
%and by the Triangle Equality, one side of the triangle $\{y,z,a\}$ has  length $D$, which contradicts the strict inequalities (\ref{eq1}) or (\ref{eq2}).
%If $ax=az+zx$ then $yz+zb+by=yz+zb+ax=yz+zb+az+zx=xy+ab$,
%and by the Triangle Equality, one side of the triangle $\{y,z,b\}$ has  length $D$, which contradicts the strict inequalities (\ref{eq1}) or (\ref{eq2}).
\end{proof}

Now we are able to present a proof of Theorem~\ref{t:main1}. Given a metric space $X$, we need to prove that $X$ is isometric to a subspace of the real line if and only if $X$ is a subline and $X$ is not an $\ell_1$-rectangle.

The ``only if'' part of this characterization is trivial. To prove the ``if'' part, assume that a metric space $X$ is a subline and $X$ is not an $\ell_1$-rectangle. If $X$ is finite, then $X$ is isometric to a subspace of the real line, by Lemma~\ref{l:Ravsky}.
It remains to consider the case of infinite metric space $X$. Pick any distinct points $a,b\in X$. Let $\mathcal F$ be the family of
all finite subsets of $X$ containing $a$ and $b$. By Lemma~\ref{l:Ravsky}, every set $F\in\mathcal F$ is isometric to a
subspace of the real line. Therefore there exists an isometry $f_F:F\to\IR$ such that $f_F(a)=0$ and $f_F(b)=ab$.
For each point $x\in F$ the image $f_F(x)$ is a unique point of $\IR$ such that $|f_F(x)|=|f_F(x)-f_F(a)|=xa$ and
$|f_F(x)-ab|=|f_F(x)-f_F(b)|=xb$. So, $f_F(x)$ is uniquely determined by the distances $xa$ and $xb$, and we can define a function $f:X\to\IR$ assigning to every $x\in X$ the real number $f_F(x)$, where $F$ is an arbitrary set in $\mathcal F$ that contains $x$.

Given any points $x,y\in X$, we can take any set $F\in\mathcal F$ with $x,y\in F$ and conclude that $|f(x)-f(y)|=|f_F(x)-f_F(y)|=xy$, which means that $f$ is an isometric embedding of $X$ into the real line.

\section{Proof of Theorem~\ref{t:main2}}

We divide the proof of Theorem~\ref{t:main2} into two lemmas.

\begin{lemma}\label{l:G1} Every nonempty  $2$-subline $X$ is isometric to a subgroup $G\subseteq \IR$ such that $\dist[X^2]=G_+$.
\end{lemma}

\begin{proof} The space $X$ is infinite, otherwise
there exist points $a,b\in X$ such that $ab=D\defeq\max\{xy:x,y\in X\}$. The Triangle Equality implies that the sphere $\Sf(a,D)$
coincides with the singleton $\{b\}$, witnessing that $X$ is not a $2$-subline.

By Theorem~\ref{t:main1}, the subline $X$ is isometric to a subspace $G$ of the real line.
Being an isometric copy of the nonempty $2$-subline $X$, the space $G$ is a nonempty $2$-subline, too.
Without loss of generality, we can assume that $0\in G\subseteq\IR$.

For every numbers $x\in G$ and $y\in  G\setminus\{0\}$, we have $|y|\in\dist[G^2]=\dist[X^2]$. Since $G$ is a $2$-subline, the sphere $\Sf(x,|y|)=G\cap \{x-|y|,x+|y|\}$ contains at least two points, which implies $\{x-y,x+y\}=\{x-|y|,x+|y|\}\subseteq G$. Consequently,  $G$ is a subgroup of $\IR$ with $\dist[X^2]=\dist[G^2]=G_+$.
\end{proof}

\begin{lemma} Let $G$ be a subgroup of the real line.
A metric space $X$ is isometric to $G$ if and only if $X$ is a $2$-subline with $\dist[X^2]=G_+$.
\end{lemma}

\begin{proof} The ``only if'' part is trivial. To prove the ``if'' part, assume that $X$ is a $2$-subline with $\dist[X^2]=G_+$. By Lemma~\ref{l:G1}, $X$ is isometric to some subgroup $H\subseteq\IR$ with $\dist[X^2]=H_+$. Then $H_+=G_+$ and hence $H=H_+\cup\{-x:x\in H_+\}=G_+\cup\{-x:x\in G_+\}=G$. Therefore, the metric space $X$ is isometric to the group $G$.
\end{proof}

\section{A metric characterization of the space $\IZ_+$}

\begin{theorem}\label{t:4.1} Let $a$ be a positive real number. A metric space $X$ is isometric to the metric space $a\IZ_+\defeq\{an:n\in\IZ_+\}$ if and only if $X$ is a ray such that $\{a,2a\}\subseteq\dist[X^2]\subseteq a\IZ_+$ and $X$ is not an $\ell_1$-rectangle.
\end{theorem}

\begin{proof} The ``only if'' part is trivial. To prove the ``if'' part, assume that $X$ is
a ray such that $\{a,2a\}\subseteq\dist[X^2]\subseteq\IZ$ and $X$ is not an $\ell_1$-rectangle. By Theorem~\ref{t:main1}, $X$ is isometric to a subspace of the real line. So, we lose no generality assuming that  $X\subseteq\IR$. Since $X$ is a ray, there exists a point $o\in X$ such that for every $r\in\dist[X^2]$, the sphere $\Sf(o,r)$ is a singleton. We lose no generality assuming that $0=o\in X\subseteq\IR$. Then $X$ is antisymmetric in the sense that for every $r\in\dist[X^2]\setminus\{0\}$ we have $r\in X$ if and only if $-r\notin X$.

It follows from $0\in X$ and $\dist[X^2]\subseteq a\IZ_+$ that $X\subseteq a\IZ$. Since the sphere $\Sf(o;a)$ is a singleton, $a\in X$ or $-a\in X$. We lose no generality assuming that $a\in X$ and hence $-a\notin X$.

We claim that $2a\in X$. To derive a contradiction, assume that $2a\notin X$ and hence $-2a\in X$, by the inclusion $2a\in\dist[X^2]$ and the antisymmetry  of $X$. Taking into account that $X$ is a $1$-subline with $-a\notin X$, $-2a\in X$ and $a\in \dist[X^2]$, we conclude that $-3a\in X$ and $3a\notin X$, by the antisymmetry of $X$. Taking into account that $X$ is a $1$-subline with $a\in X$, $-a\notin X$,  and $2a\in\dist[X^2]$, we conclude that $a+2a=3a\in X$, which is a desired contradiction showing that $2a\in X$ and $-2a\notin X$, by the antisymmetry of $X$. The following lemma and the antisymmetry of $X$ complete the proof of the theorem.
\end{proof}

\begin{lemma} Let $a$ be a positive real number. If $X\subset a\IZ$ is a $1$-subline with $0,a,2a\in X$ and $-2a,-a\notin X$, then $a\IZ_+\subseteq X$.
\end{lemma}
\begin{proof}
%{\red Now we can provide a shorter proof. Indeed, by Theorem~\ref{t:semiaffine} 
%$X=(H+c)\cup (H+b)$ for some subgroup $H$ of $G$ and some elements $c,b\in X$ or 
%$X=(H\setminus C)+g$ for some $g\in G$, some subgroup $H\subseteq G$ and some midconvex set $C$ in $H$.
%In the first case one of the cosets $H+$ contains $0$, so the  
%}
By induction we shall prove that for every $i\in\IN$ the set $\{0,a,\dots,ia\}$ is a subset of $X$. This is so for $i\in\{1,2\}$. Assume that for some positive number $i\ge 2$ we know that  $\{0,a,\dots,ia\}\subseteq X$.

 If $i$ is even, consider the number $c=\frac12ia$ and observe that $c+a=\frac12ai+a\le \frac12ia+\frac12ia=ia$ and hence $c,c+a\in \dist[X^2]$. Since $X$ is a $1$-subline and $c-(c+a)=-a\notin X$, the number $c+(c+a)=ia+a$ belongs to $X$, witnessing that $\{0,a,\dots,ia,(i+1)a\}\subseteq X$.

 If $i$ is odd, consider the number $c=\frac12(i-1)a$ and observe that $c+2a=\frac12ia+\frac32a\le \frac12ia+\frac12ia=ia$ and hence $c,c+2a\in \dist[X^2]$.  Since $X$ is a $1$-subline and $c-(c+2a)=-2a\notin X$, the number $c+(c+2a)=(i+1)a$ belongs to $X$, witnessing that $\{0,a,\dots,ia,(i+1)a\}\subseteq X$.

This completes the inductive step. By the Principle of Mathematical Induction, $\{0,a,\dots,ia\}\subseteq X$ and hence $a\IZ_+\subseteq X$.
\end{proof}

\begin{remark} The ``if'' part of Theorem~\ref{t:4.1} can be also derived from the properties of semiaffine and midconvex sets established in Theorems~\ref{t:1sub} and \ref{t:3}. Indeed, assume that a metric space $X$ is a ray such that $\{a,2a\}\subseteq\dist[X^2]\subseteq a\IZ_+$ and $X$ is not an $\ell_1$-rectangle. Applying Theorem~\ref{t:1sub}, we can prove that $X$ is isometric to $H\setminus C$  for some subgroup $H$ of $\IR$, some midconvex set $C$ in the group $H$. 
Composing the isometry with a shift on the group $H$ we can assume that $0\in H\setminus C$ and
for each $r\in\dist[X^2]$, the sphere $\Sf(0,r)$ is a singleton. 
Since $\dist[X^2]\subseteq a\IZ_+$, replacing $H$ by $H\cap a\IZ$ and $C$ by $C\cap a\IZ$, if needed, we can suppose that $H=a\IZ$. By Theorem~\ref{t:3}, $C=C'\cap(P+x)$ for some order-convex set $C'\subseteq H=a\IZ$, some $x\in C$ and some subgroup $P$ of $H$. %such that the quotient group $H/P$ contains no elements of even order.
Since for every $r\in\dist[X^2]$ the sphere $\Sf(0,r)$ is a singleton, the coset $P+x$ equals $H$, and $C'\cap (P+x)$ equals $H\cap a\IN$ or $H\cap (-a\IN)$.
\end{remark}

\section{Proof of Theorem~\ref{t:main3}} Let $G$ be a subgroup of the additive group $\IQ$ of rational numbers. Given a metric space $X$, we should prove that $X$ is isometric to the monoid $G_+\defeq\{x\in G:x\ge 0\}$ if and only if $X$ is a ray with $\dist[X^2]=G_+$. The ``only if'' part of this characterization is trivial. To prove the ``if'' part, assume that $X$ is a ray with $\dist[X^2]=G_+$. If the group $G$ is trivial, then $\dist[X^2]=G_+=\{0\}$ and hence $X$ is a singleton, isometric to the singleton $G_+=\{0\}$. So, assume that the group $G$ is not trivial. Then $G$ is infinite and so is the set $G_+$. Since $\dist[X^2]=G_+$, the metric space $X$ is infinite and hence is not an $\ell_1$-rectangle.% By Theorem~\ref{t:main1}, $X$ is isometric to a subspace of the real line. We lose no generality assuming that $X\subseteq\IR$Being a ray, the space $X$ contains a point $o\in X$ sui

If the group $G$ is finitely generated, then $G$ is cyclic (being a subgroup of a cyclic subgroup of $\IQ$) and hence $G=a\IZ$ for some positive number $a\in G\subseteq\IQ$. By Theorem~\ref{t:4.1}, $X$ is isometric to $a\IZ_+=G_+$.

If the group $G$ is infinitely generated, then $G=\bigcup_{n\in\w}G_n$ for a strictly increasing sequence $(G_n)_{n\in\w}$ of non-trivial finitely-generated subgroups $G_n$ of $G\subseteq\IQ$.
Every group $G_n$ is cyclic, being a subgroup of a suitable cyclic subgroup of the group $\IQ$. Then $G_n=a_n\IZ$ for some positive rational number $a_n$.

Being a ray, the space $X$ contains a point $o\in X$ such that for every $r\in \dist[X^2]$, the sphere $\Sf(o;r)$ is a singleton. By Theorem~\ref{t:main1}, the infinite subline $X$ is isometric to a subspace of the real line. We lose no generality assuming that $0=o\in X\subseteq\IR$. It follows from $\dist[X^2]=G_+\subseteq\IQ$ that $X\subseteq \IQ$. It is easy to see that for every $n\in\IN$ the intersection $X_n=X\cap G_n$ is a ray such that $\dist[X_n^2]=(G_n)_+=a_n\IZ_+$. By Theorem~\ref{t:4.1}, the space $X_n$ is isometric to the space $a_n\IZ_+$ and hence $X_n=a_n\IZ_+$ or $X_n=-a_n\IZ_+$. We lose no generality assuming that $X_0=a_0\IZ_+$. Then $X_n=a_n\IZ_+$ for all $n\in\w$ and hence $X=\bigcup_{n\in\w}X_n=\bigcup_{n\in\w}a_n\IZ_+=\bigcup_{n\in\w}(G+n)_+=G_+$.

\section{Proof of Theorem~\ref{t:main4}}

Given a metric space $X$, we should prove that $X$ is isometric to the half-line $\IR_+$ if and only if $X$ is a complete ray such that $\IQ_+\subseteq\dist[X^2]$. The ``only if'' part of this characterization is trivial. To prove the ``if'' part, assume that $X$ is a complete ray such that $\IQ_+\subseteq\dist[X^2]$. Then $X$ is infinite and hence is isometric to a subspace of the real line. Being a ray, the space $X$ contains a point $o\in X$ such that for every $r\in\dist[X^2]$ the sphere $\Sf(o;r)$ is a singleton. We lose no generality assuming that $0=o\in X\subseteq \IR$.

 It is easy to see that the subspace $Y=\{x\in X:xo\in \IQ_+\}$ of $X$ is a ray with $\dist[Y^2]=\IQ_+$. By Theorem~\ref{t:main3}, $Y$ is isometric to $\IQ_+$. Then $Y=\IQ_+$ or $Y=-\IQ_+$. Replacing the set $X\subseteq\IR$ by $-X$, if necessary, we can assume that $Y=\IQ_+$. Being complete, the metric space $X$ contains the completion of its subspace $Y$, which coincides with $\IR_+$. Then $\IR_+\subseteq X$ and $\dist[X^2]=\IR_+$. Assuming that $X\ne\IR_+$, we can find a point $x\in X\setminus\IR_+$ and conclude that the sphere $\Sf(o;|x|)$ contains two distinct points $x$ and $-x$, which contradicts the choice of the point $o$.

\section{Constructing  Example~\ref{ex:1}}\label{s:field}

In this section we elaborate the construction of the ray $X$ from Example~\ref{ex:1}.

This construction exploits the following algebraic characterization of rays in the real line.

\begin{lemma}\label{l:ray} A subset $X$ of the real line is a ray if and only if it has two properties:
\begin{enumerate}
\item $\forall x\in X\;\forall r\in X-X\;\{x-r,x+r\}\cap X\ne\emptyset$;
\item $\exists o\in X\;\forall r\in X-X\;|\{o-r,o+r\}\cap X|\le 1$.
\end{enumerate}
\end{lemma}

\begin{proof} To see that this characterization indeed holds, observe that
$\dist[X^2]=(X-X)_+$ and $\Sf(x;r)=\{x-r,x+r\}\cap X$ for any $x\in X$ and $r\in\IR_+$.
\end{proof}

A function $f:\IR\to\IR$ is called {\em additive} if $f(x+y)=f(x)+f(y)$ for all $x,y\in\IR$.

\begin{lemma}\label{l:im-ray} For every ray $X\subseteq \IR$ and every injective additive function $f:\IR\to\IR$, the metric subspace $f[X]$ of $\IR$ is a ray.
\end{lemma}

\begin{proof} To show that $f[X]$ is a ray, it suffices to check that $f[X]$ satisfies the algebraic conditions of Lemma~\ref{l:ray}.
\smallskip

1. Fix any numbers $y\in f[X]$ and $r\in f[X]-f[X]$. The additivity of $f$ ensures that $f[X]-f[X]=f[X-X]$ and hence $r=f(s)$ for some $s\in X- X$. Since $X$ is a ray, for the element $x=f^{-1}(y)\in X$, the set $X\cap\{x-s,x+s\}$ contains some point $z=x\pm s$. The additivity of $f$ ensures that $f(z)=f(x\pm s)=f(x)\pm f(s)=y\pm r\in\{y-r,y+r\}\cap f[X]$ and hence the set $\{y-r,y+r\}\cap f[X]$ is not empty.
\smallskip

2. Since $X$ is a ray, there exists a point $o$ such that for every $s\in \dist[X^2]=(X-X)_+$ the sphere $\{x\in X:|o-x|=s\}$ is a singleton. We claim that the point $f(o)$ has the property required in condition (2) of Lemma~\ref{l:ray}. Assuming that this condition does not hold, we can find a real number $r\in f[X]-f[X]$ such that $\{f(o)-r,f(o)+r\}\cap f[X]$ is a doubleton. Then $f(o)-r=f(x)$ and $f(o)+r=f(y)$ for some distinct real numbers $x,y\in X$. Since $r\in f[X]-f[X]=f[X-X]$, the real number $s=f^{-1}(r)$ belongs to the set $X-X$ and hence $|s|\in\dist[X^2]$.

Since the function $f$ is additive, $f(o-s)=f(o)-f(s)=f(o)-r=f(x)$ and hence $o-s=x$ by the injectivity of $f$. By analogy we can show that $o+s=y$. Then $\{x,y\}=\{o-s,o+s\}=\{o-|s|,o+|s|\}$ is a doubleton in $X$, which contradicts the choice of the point $o$. This contradiction completes the proof of the second condition of Lemma~\ref{l:ray} for the set $f[X]$.
\smallskip

By Lemma~\ref{l:ray}, the metric subspace $f[X]$ of the real line is a ray.
\end{proof}

Now we are able to justify Example~\ref{ex:1}. Let $G$ be any subgroup of $\IR$ containing two nonzero elements $a,b\in G$ such that $b\notin\IQ\cdot a\subseteq G$. We lose no generality assuming that $a,b>0$. Consider the real line as a vector space $\IQ$ over the field of rational numbers. The conditions $b\notin\IQ{\cdot}a$ and $a\ne 0$ imply that the vectors $a,b$ are linearly independent over the field $\IQ$. Using the Kuratowski--Zorn Lemma, choose a maximal subset $B\subseteq \IR$ such that $\{a,b\}\subseteq B$ and $B$ is linearly independent over $\IQ$. The maximality of $B$ guarantees that $B$ is an algebraic basis of the vector space $\IR$ over the field $\IQ$. Consider the function $f:B\to\IR$ such that $f(a)=-a$ and $f(x)=x$ for all $x\in B\setminus\{a\}$. By the linear independence of $B$, the function admits a unique extension to an additive function $\bar f:\IR\to\IR$. Since the set $f[B]=\{-a\}\cup(B\setminus\{a\})$ is linearly independent over the field $\IQ$, the additive function $\bar f:\IR\to\IR$ is injective.

\begin{claim}\label{cl:fG=G} $f[G]=G$.
\end{claim}

\begin{proof} Since $B$ is a basis of the $\IQ$-vector space $\IR$, for every element $x\in G$, there exist a finite set $F\subseteq B$ and a function $\lambda:F\to\IQ\setminus\{0\}$ such that $x=\sum_{e\in F}\lambda(e)\cdot e$. If $a\notin F$, then $\bar f(x)=\sum_{e\in F}\lambda(e)f(e)=\sum_{e\in F}\lambda(e)e=x\in G$.
If $a\in F$, then
\begin{multline*}
\bar f(x)=\lambda(a)f(a)+\sum_{e\in F\setminus\{a\}}\lambda(e)f(e)=-\lambda(a)a+\sum_{e\in F\setminus\{a\}}\lambda(e)e\\
=-2\lambda(a)a+\sum_{e\in F}\lambda(e)e=-2\lambda(a)a+x\in G
\end{multline*} because $\IQ{\cdot}a\subseteq G$.

Therefore, $\bar f[G]\subseteq G$. Since $f\circ f$ is the identity map of the algebraic  basis $B$ of the $\IQ$-vector space $\IR$, the composition $\bar f\circ\bar f$ is the idenity function of $\IR$. Then $\bar f[G]\subseteq G$ implies $G=\bar f[\bar f[G]]\subseteq \bar f[G]$ and hence $\bar f[G]=G$.
\end{proof}

 By Lemma~\ref{l:im-ray}, the metric subspace $X=f[G_+]$ of the real line is a ray.
 Since $f$ is additive, $X-X=f[G_+]-f[G_+]=f[G_+-G_+]=f[G]=G$, see Claim~\ref{cl:fG=G}. Then $\dist[X^2]=(X-X)_+=G_+$.

Taking into account that the numbers $a,b$ are positive and $\bar f(a)=f(a)=-a$, $\bar f(b)=f(b)=b$, we conclude that the set $\IZ_+b-\IQ_+a\subseteq f[G_+]+f[G_+]=f[G_++G_+]=f[G_+]=X$ is dense in $\IR$.

Since the set $X$ is dense in $\IR$, the completion of the metric space $X$ coincides with $\IR$. On the other hand, the completion of the space $G_+\supseteq \IQ_+{\cdot}a$ coincides with the half-line $\IR_+$, which is not isometric to $\IR$. This implies that the ray $X$ is not isometric to the ray $G_+$.

\section{Acknowledgement}

The second author express his sincere thanks to Mamuka Jibladze for pointing out the references \cite{Bowers} and \cite{Menger} to Menger's Theorems~\ref{t:main1} and \ref{t:Rn}, and also to Pavlo Dzikovskyi for fruitful discussions on the Triangle Equality in metric spaces.

\end{document}